\documentclass[11pt,letterpaper,reqno]{amsart}
\usepackage[T1]{fontenc}
\usepackage{lmodern}
\usepackage{amsmath,amssymb,amsthm}
\usepackage[margin=1in,headheight=12pt,headsep=18pt]{geometry}
\usepackage{graphicx}
\usepackage{float}
\usepackage[hidelinks]{hyperref}
\hypersetup{pdftitle={A Counterexample to Three Conjectures of Luo},pdfauthor={Qilong Guo}}
\newtheorem{theorem}{Theorem}[section]
\newtheorem{lemma}[theorem]{Lemma}
\newtheorem{proposition}[theorem]{Proposition}
\theoremstyle{remark}
\newtheorem{remark}[theorem]{Remark}
\newcommand{\T}{\mathcal T}
\newcommand{\R}{\mathbb R}
\newcommand{\C}{\mathbb C}
\newcommand{\Z}{\mathbb Z}
\newcommand{\SAS}{\operatorname{SAS}}
\DeclareMathOperator{\vol}{vol}
\DeclareMathOperator{\supp}{supp}
\title{A Counterexample to Three Conjectures of Luo}
\author{Qilong Guo}
\date{}
\makeatletter
\@ifundefined{subjclassname@2020}{%
  \@namedef{subjclassname@2020}{\textup{2020} Mathematics Subject Classification}}{}
\makeatother

\begin{document}
\begin{abstract}
We give a single counterexample to Conjectures 1--3 as stated in
Luo's paper \cite{Luo2011}: a minimal two-tetrahedron triangulation of
the lens space $L(3,1)$ with a null-homotopic loop edge.
The proof uses quadrilateral matching equations, an explicit
calculation of shape parameters, and known results on volume maximization.
\end{abstract}
\maketitle

\section{Introduction}
Luo proposed the following conjectures for a minimal triangulation
$\T$ of a closed $3$-manifold $M$ \cite[Conjectures 1--3]{Luo2011}.
A triangulation is minimal if it uses the fewest tetrahedra among
all triangulations of the same manifold.
As in \cite{Luo2011}, we allow triangulations obtained by affine
face pairings of tetrahedra, including self-identifications.
Here $M$ is assumed irreducible in (C1) and (C2), and orientable in
(C1) and (C3).
\begin{enumerate}
\item[(C1)] $\T$ supports either a solution to Thurston's equation
or a cluster of three 2-quad-type solutions.
\item[(C2)] If every maximum of $\vol\colon\SAS(\T)\to\R$ is
smooth, then $\T$ supports a solution to Thurston's equation.
\item[(C3)] If $\T$ has a null-homotopic edge whose endpoints agree,
then $\T$ supports a cluster of three 2-quad-type solutions.
\end{enumerate}
Here $\SAS(\T)$ is the space of circle-valued angle structures,
and $\vol$ is its volume function. These and the other definitions
are recalled in Sections~\ref{sec:quadrilaterals}
and~\ref{sec:conjectures}. We prove the following.

\begin{theorem}\label{thm:main}
There is a minimal two-tetrahedron triangulation $\T$ of $L(3,1)$
such that
\begin{enumerate}
\item[(i)] $\T$ has a null-homotopic loop edge;
\item[(ii)] $\T$ supports no cluster of three 2-quad-type solutions;
\item[(iii)] $\T$ supports no solution to Thurston's equation;
\item[(iv)] every maximum of $\vol\colon\SAS(\T)\to\R$ is smooth.
\end{enumerate}
Consequently, $\T$ is a counterexample to (C1), (C2), and (C3).
\end{theorem}

\begin{remark}\label{rmk:scope}
These conclusions concern the conjectures as stated.
Since $L(3,1)$ is not simply
connected, the example does not address the restricted versions used
in Luo's proposed approach to the Poincar\'e conjecture.
\end{remark}

The paper is organized as follows. In Section~\ref{sec:triangulation},
we give the face pairings defining the triangulation of $L(3,1)$,
establish its minimality, and exhibit a null-homotopic loop edge.
In Section~\ref{sec:quadrilaterals}, we compute the signed
quadrilateral matching matrix and use it to exclude clusters of three
2-quad-type solutions. Section~\ref{sec:conjectures} completes the
proof of Theorem~\ref{thm:main}: Section~\ref{subsec:thurston}
establishes the nonexistence of solutions to Thurston's equation,
and Section~\ref{subsec:volume} proves the smoothness of every
volume maximizer.

\section{The two-tetrahedron triangulation}\label{sec:triangulation}
Take two ordered tetrahedra $A$ and $B$, each with vertex set
$\{0,1,2,3\}$. Write $A(ijk)$ for the face of $A$ with vertices
$i,j,k$ in the displayed order, and similarly for $B$. In each row
below, corresponding vertices are identified:
\begin{equation}\label{eq:pairings}
\begin{array}{ccc}
\text{first ordered face}&&\text{second ordered face}\\[2pt]
A(123)&\longleftrightarrow&A(230)\\
B(013)&\longleftrightarrow&B(120)\\
A(013)&\longleftrightarrow&B(321)\\
A(012)&\longleftrightarrow&B(320).
\end{array}
\end{equation}
Denote the quotient by $M$ and its induced triangulation by $\T$.
The same data are shown schematically in Figure~\ref{fig:pairings}.

\begin{figure}[H]
\centering
\includegraphics[width=0.48\textwidth]{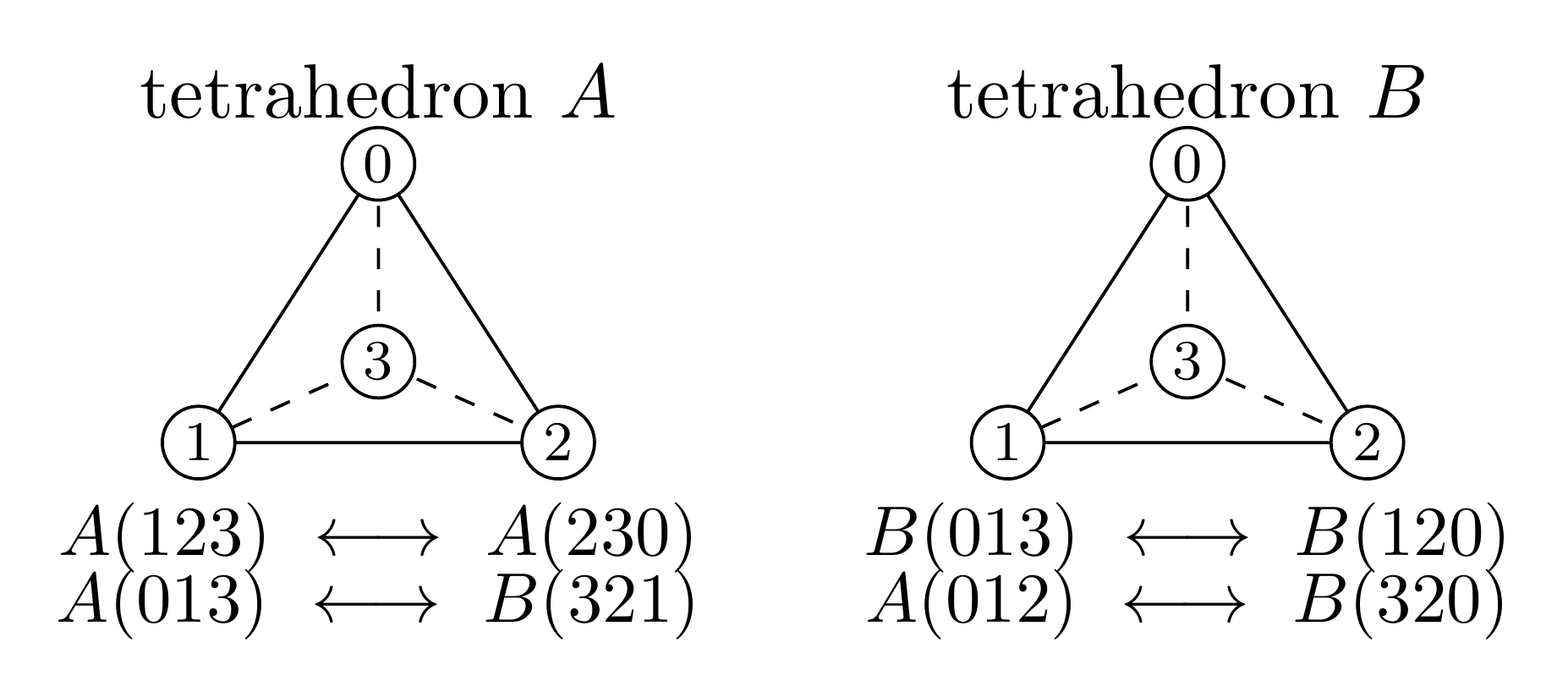}
\caption{The two ordered tetrahedra and the face pairings defining
the triangulation. Dashed segments are the edges hidden behind the
front face $012$. In $X(ijk)\leftrightarrow Y(i'j'k')$, the vertices
$i,j,k$ are identified with $i',j',k'$, respectively.}
\label{fig:pairings}
\end{figure}

The four face maps in \eqref{eq:pairings} extend, in the order listed,
to the vertex permutations $1230$, $1230$, $3201$, and $3201$.
Each four-digit word lists the images of $0,1,2,3$, in that order.
Each is odd, so the local orientations of $A$ and $B$ are compatible
with an orientation of the quotient. The identifications put all eight
tetrahedral vertices in one class, and no local edge is identified with
itself with reversed orientation. The vertex link is connected and consists
of eight triangles, with twelve edges and six vertices after
identification. Its Euler characteristic is therefore $6-12+8=2$.
The link is a closed orientable surface, hence is $S^2$.
Thus the quotient is a closed orientable $3$-manifold with one vertex.

There are three global edges, denoted $a,b,c$. The six local edges
of each tetrahedron are identified as follows:
\begin{equation}\label{eq:edges}
\begin{array}{c|cccccc}
 &01&02&03&12&13&23\\\hline
A&a&b&c&c&b&c\\
B&b&c&b&b&c&a
\end{array}
\end{equation}
Hence the degrees of $a,b,c$, counting local edge occurrences,
are $2,5,5$, respectively.
Since there is one vertex, the oriented edges $a,b,c$ generate
$\pi_1(M)$. Orient them by $A(01)$, $A(02)$, and $A(03)$,
respectively, with each local edge directed from its first vertex to
its second. Reading the boundaries of $A(123)$, $A(013)$, $A(012)$,
and $B(013)$ gives the presentation
\begin{equation}\label{eq:pi1}
\pi_1(M)=\langle a,b,c\mid
c^{-2}b,\ ab^{-1}c^{-1},\ ac^{-1}b^{-1},\ bc^{-1}b\rangle.
\end{equation}
The first and fourth relations give $b=c^2$ and $c=b^2$, while the
second gives $a=cb$. Substitution yields
\begin{equation}\label{eq:group}
\pi_1(M)\cong\langle c\mid c^3=1\rangle\cong\Z/3,
\qquad a=c^3=1.
\end{equation}
The face pairings in \eqref{eq:pairings} give the two-tetrahedron
triangulation of $L(3,1)$ shown in
\cite[Figure 3(b); see also Proposition 11(5b)]{JRT}.
To identify the diagrams explicitly, take $A$ and $B$ to be the left
and right tetrahedra in that figure. Their vertices $(0,1,2,3)$ occupy
the positions (top, bottom, right, left) and (left, right, bottom, top),
respectively. The edges $a,b,c$ correspond to the edges marked with
one, three, and two arrowheads, respectively.
Equation~\eqref{eq:group} also shows that the loop edge $a$ is
null-homotopic.

The minimum number of tetrahedra in a triangulation of $L(3,1)$
is two \cite[Introduction]{JRT}. Thus $\T$ is minimal.

\section{The quadrilateral matching calculation}\label{sec:quadrilaterals}
A normal surface meets each tetrahedron in triangular and
quadrilateral disks, whose counts are its normal coordinates.
A normal arc in a triangular face joins two distinct edges; its
type is determined by the vertex it cuts off.
Haken's matching equations equate the total
disk coordinates contributing to each arc type on the two sides
of a paired face. We consider real solutions $z$ of this homogeneous
linear system and write $z(q)$ for the coordinate of quadrilateral
type $q$. Such a solution is of
\emph{2-quad type} if exactly one or two of its quadrilateral
coordinates are nonzero. A \emph{cluster of three 2-quad-type
solutions} consists of three such solutions $z_1,z_2,z_3$, not
necessarily distinct, and a tetrahedron $\sigma$ with three distinct quadrilateral
types $q_1,q_2,q_3$ such that
\[
z_i(q_i)\ne0\qquad(i=1,2,3).
\]
We impose neither nonnegativity nor the restriction to one
quadrilateral type per tetrahedron.

In $A$, denote the quadrilateral types disjoint from the opposite-edge pairs
\[
01/23,\qquad 02/13,\qquad 03/12
\]
by $x_1,x_2,x_3$, respectively. Use $y_1,y_2,y_3$ for the
corresponding types in $B$. Give both tetrahedra the orientations
determined by their vertex orders and use the cyclic orders
\[
x_1\longrightarrow x_2\longrightarrow x_3\longrightarrow x_1,
\qquad
y_1\longrightarrow y_2\longrightarrow y_3\longrightarrow y_1.
\]
Let $Q(\T)$ be the set of all six quadrilateral types, and write
$q\subset\sigma$ when $q$ belongs to the tetrahedron $\sigma$.
To record the cyclic order, set $w(q,q')=1$ if $q\to q'$ and
$w(q,q')=-1$ if $q'\to q$; set it to zero otherwise.
Let $i(e,q)\in\{0,1,2\}$ count how many of the two local edges
disjoint from $q$ are identified with the global edge $e$,
counting each local edge separately.
Following Luo \cite[Lemma 3.2 and equation (3.9)]{Luo2011},
the \emph{quadrilateral matching equation} (or \emph{$Q$-matching
equation}) at $e$ for a quadrilateral coordinate vector
$v=(v(q))_{q\in Q(\T)}$ is
\[
\sum_{q\in Q(\T)}W(e,q)v(q)=0,
\qquad
W(e,q)=\sum_{q'\in Q(\T)}i(e,q')w(q',q).
\]
Write $W=(W(e,q))$, with rows ordered as $a,b,c$ and columns as
$x_1,x_2,x_3,y_1,y_2,y_3$. For any solution $z$ of Haken's equations,
its \emph{quadrilateral projection} is the column vector
$v=(z(q))_{q\in Q(\T)}$, which satisfies $Wv=0$ \cite{Tollefson}.
With $I=(i(e,q))$ and $\Omega=(w(q,q'))$, these orders give
\[
I=
\begin{pmatrix}
1&0&0&1&0&0\\
0&2&0&1&0&2\\
1&0&2&0&2&0
\end{pmatrix},
\qquad
\Omega=\operatorname{diag}(J,J),
\quad
J=\begin{pmatrix}0&1&-1\\-1&0&1\\1&-1&0\end{pmatrix}.
\]
Hence
\begin{equation}\label{eq:matching}
W=I\Omega=\begin{pmatrix}
0&1&-1&0&1&-1\\
-2&0&2&2&-1&-1\\
2&-1&-1&-2&0&2
\end{pmatrix}.
\end{equation}
For a quadrilateral coordinate vector $v$, write
$\supp(v)=\{q\in Q(\T):v(q)\ne0\}$ for its support.
Thus $\supp(v)$ records exactly those quadrilateral types whose
coordinates in $v$ are nonzero.

\begin{lemma}\label{lem:support}
Let $v\ne0$ satisfy $Wv=0$ and have at most two nonzero
coordinates. Then $\supp(v)=\{x_1,y_1\}$. In fact, $v$ is a
nonzero multiple of $(1,0,0,1,0,0)^{\mathsf T}$.
In particular, $\T$ supports no cluster of three 2-quad-type solutions.
\end{lemma}
\begin{proof}
A kernel vector supported on one coordinate would give a zero
column of $W$, and one supported on two coordinates would give a
proportional pair of columns. In \eqref{eq:matching} no column is
zero, and the only proportional pair consists of the $x_1$- and
$y_1$-columns, which are negatives of one another. This proves
the assertion about $v$.

Let $z$ be a 2-quad-type solution and let $v$ be its quadrilateral
projection. Then $v\ne0$, $Wv=0$, and $v$ has at most two nonzero
coordinates, so the preceding calculation gives
$\supp(v)=\{x_1,y_1\}$. Thus
every 2-quad-type solution vanishes at $x_2,x_3,y_2,y_3$, so
neither tetrahedron supplies the three distinguished quadrilateral
types required by a cluster. Hence no such cluster exists.
\end{proof}

\section{The three conjectures}\label{sec:conjectures}
We now establish the remaining properties in Theorem~\ref{thm:main}.

\subsection{Thurston's equation}\label{subsec:thurston}
We use the quadrilateral types $Q(\T)$ and incidence numbers
$i(e,q)$ defined in Section~\ref{sec:quadrilaterals}; the latter
are the entries of the matrix $I$.

A solution to Thurston's equation \cite[Section 3.2]{Luo2011}
is a shape assignment
\[
\zeta\colon Q(\T)\longrightarrow\C\setminus\{0,1\},
\]
where $\zeta(q)$ is the complex shape parameter associated to $q$.
For the three types $q,q',q''$ in a tetrahedron, listed in the
cyclic order $q\to q'\to q''\to q$ fixed in
Section~\ref{sec:quadrilaterals}, the parameters satisfy
\[
\bigl(\zeta(q),\zeta(q'),\zeta(q'')\bigr)
=\left(s,\frac{1}{1-s},\frac{s-1}{s}\right),
\qquad s=\zeta(q).
\]
In addition, for every global edge $e$ one requires
\[
\prod_{q\in Q(\T)}\zeta(q)^{i(e,q)}=1.
\]
The exponent records the multiplicity with which a shape parameter
occurs around $e$. These are the ordinary gluing equations; for the
triangulation at hand, the global edges are $a,b,c$ from
Section~\ref{sec:triangulation}.

\begin{proposition}\label{prop:thurston}
The triangulation $\T$ supports no solution to Thurston's equation.
\end{proposition}
\begin{proof}
Suppose that such a shape assignment $\zeta$ exists, and put
$z=\zeta(x_1)$ and $w=\zeta(y_1)$. Here $x_i$ and $y_i$ are the
quadrilateral types in $A$ and $B$ defined in
Section~\ref{sec:quadrilaterals}. In particular,
$z,w\in\C\setminus\{0,1\}$, and
\[
\begin{aligned}
\bigl(\zeta(x_1),\zeta(x_2),\zeta(x_3)\bigr)
&=\left(z,\frac{1}{1-z},\frac{z-1}{z}\right),\\
\bigl(\zeta(y_1),\zeta(y_2),\zeta(y_3)\bigr)
&=\left(w,\frac{1}{1-w},\frac{w-1}{w}\right).
\end{aligned}
\]
The edge equations at $a,b,c$, read from the incidence matrix in
Section~\ref{sec:quadrilaterals}, are
\begin{align*}
zw&=1,\\
\left(\frac{1}{1-z}\right)^2w
\left(\frac{w-1}{w}\right)^2&=1,\\
z\left(\frac{z-1}{z}\right)^2
\left(\frac{1}{1-w}\right)^2&=1.
\end{align*}
The first gives $w=z^{-1}$. Substituting into the last two reduces
their left-hand sides to $z^{-1}$ and $z$, respectively. They force
$z=1$, which is forbidden. Hence no solution exists.
\end{proof}

\begin{remark}\label{rmk:generalized-thurston}
Luo's generalized Thurston equation \cite[Section 3.2]{Luo2011}
allows edge products $\pm1$. Taking $z=w=-1$ gives the shape triple
$(-1,\tfrac12,2)$ in each tetrahedron, with edge products $(1,-1,-1)$
at $(a,b,c)$. Thus the generalized equation has a solution, and the
example does not contradict \cite[Theorem 1.2]{Luo2011}.
\end{remark}

\begin{remark}\label{rmk:thurston-theorem}
Proposition~\ref{prop:thurston} also follows from the
Segerman--Tillmann theorem \cite[Theorem 1.1]{Luo2011}, since
$a$ is a null-homotopic loop edge by
Section~\ref{sec:triangulation}.
\end{remark}

\subsection{Circle-valued angle structures and volume}\label{subsec:volume}
A \emph{circle-valued angle structure} is a map
$u\colon Q(\T)\to S^1$, where $S^1\subset\C$ is the unit circle,
satisfying
\[
\begin{aligned}
\prod_{q\subset\sigma}u(q)&=-1&&\text{for every tetrahedron }\sigma,\\
\prod_{q\in Q(\T)}u(q)^{i(e,q)}&=1&&\text{for every edge }e.
\end{aligned}
\]
Their space is denoted by $\SAS(\T)$. It is nonempty by
\cite[Proposition 4.1]{Luo2011} and compact, being a closed
subset of $(S^1)^6$. Using the Lobachevsky function
$\Lambda$, the volume is defined by
\[
\vol(u)=\sum_{q\in Q(\T)}\Lambda(\arg u(q)),
\qquad
\Lambda(\theta)=-\int_0^\theta\log|2\sin t|\,dt.
\]
Here $\arg u(q)$ is any real $\theta$ with $u(q)=e^{i\theta}$.
The function $\Lambda$ is odd and $\pi$-periodic; the periodicity
makes the sum independent of these choices.
The volume function is continuous, so it attains a
maximum. Following Luo's terminology, a point $u$ is \emph{smooth}
if $u(q)\notin\{1,-1\}$ for every $q\in Q(\T)$.

\begin{proposition}\label{prop:volume}
Every maximum of $\vol\colon\SAS(\T)\to\R$ is smooth.
\end{proposition}
\begin{proof}
Suppose that a non-smooth maximum exists. By the Futer--Gu\'eritaud
theorem, as stated in \cite[Theorem 4.4]{Luo2011}, we may choose a
non-smooth maximum $\widetilde u$ such that, in each tetrahedron,
if one quadrilateral parameter belongs to $\{1,-1\}$, then all
three do. Since $\widetilde u$ is non-smooth, there is a tetrahedron
$\sigma$ whose three quadrilateral types $q_1,q_2,q_3$ satisfy
$\widetilde u(q_i)\in\{1,-1\}$ for $i=1,2,3$.
By \cite[Theorem 4.2(a)]{Luo2011}, each quadrilateral taking the value
$\pm1$ at such a maximum has a nonzero coordinate in some
2-quad-type solution of Haken's equations. Applying this result to
$q_1,q_2,q_3$ gives solutions $z_1,z_2,z_3$ with $z_i(q_i)\ne0$.
These form a cluster, contradicting
Lemma~\ref{lem:support}. Hence every maximum is smooth.
\end{proof}

\begin{proof}[Proof of Theorem~\ref{thm:main}]
Section~\ref{sec:triangulation} constructs $\T$ as a minimal
triangulation of $L(3,1)$ and proves (i);
Lemma~\ref{lem:support} proves (ii); and
Propositions~\ref{prop:thurston} and~\ref{prop:volume} prove
(iii) and (iv). The manifold $L(3,1)$ is closed, orientable, and
irreducible, so all the required hypotheses hold. Now (ii)--(iii)
refute (C1), (iii)--(iv) refute (C2), and (i)--(ii) refute (C3).
\end{proof}

\section*{Declaration on the use of AI}
The author acknowledges the use of ChatGPT (OpenAI) for providing
the idea for the counterexample and for assistance with
English-language polishing. All content of this paper has been
carefully reviewed and verified by the author. The author takes full
responsibility for the entire manuscript, including its mathematical
arguments, calculations, conclusions, and any errors or omissions.


\begin{thebibliography}{3}
\bibitem{JRT}
W.~Jaco, J.~H.~Rubinstein and S.~Tillmann,
\emph{Minimal triangulations for an infinite family of lens spaces},
J. Topol. \textbf{2} (2009), no.~1, 157--180.

\bibitem{Luo2011}
F.~Luo, \emph{Triangulated 3-manifolds: from Haken's normal
surfaces to Thurston's algebraic equation}, in
\emph{Interactions between Hyperbolic Geometry, Quantum Topology
and Number Theory}, Contemp. Math. \textbf{541}, Amer. Math. Soc.,
Providence, RI, 2011, 183--204.

\bibitem{Tollefson}
J.~L.~Tollefson, \emph{Normal surface $Q$-theory}, Pacific J. Math.
\textbf{183} (1998), no.~2, 359--374.
\end{thebibliography}
\end{document}